\documentclass[a4paper,12pt]{article}
\title{\bf Local Functions on Blocks}
\author{ Jamie Mason, \ JXM1289@bham.ac.uk, \\ School of Mathematics, University of Birmingham}
\usepackage{preamble}
\usepackage{import}

\begin{document}

\maketitle

\begin{abstract}
    We define a block-by-block version of Isaacs and Navarro's chain local condition and then prove that the Alperin--McKay conjecture is equivalent to a certain function on groups having this property. We then go on to prove several other block-by-block versions of results from Isaacs and Navarro's paper.
\end{abstract}

\section{Introduction}

We define an integer valued function, \(f\), on pairs \((G,B)\), where \(G\) is a group and \(B\) is a \(p\)-block of \(G\), to be {\it block chain local} if
\[
    \sum_{C \in \RR} (-1)^{|C|} \sum_{b \mid B_C}f(G_C,b)=0,
\]
where the first sum runs over representatives of \(G\)-orbits of chains of \(p\)-subgroups in \(G\), \(G_C\) is the stabiliser group of \(C\) and \(B_C\) is a sum of blocks of \(G_C\). This definition is motivated as a block-by block version of the definition of local functions give by Isaacs and Navarro \cite{isa20}. In their paper they show that the McKay conjecture is equivalent to a function on groups having this local property and here we extend this result to show that the Alperin--McKay conjecture is equivalent to a function on pairs being block chain local. 

As Isaacs and Navarro note, their work is not the first occurrence of a local global conjecture being restated as a condition on chains of subgroups. This is partially as a result of the local-global conjectures' resistance to being proved in full generality, with the one exception being Brauer's height zero conjecture \cite{mal22}, so effort has been directed towards finding equivalent statements to the conjectures. Perhaps the most well known example of this is Kn\"orr and Robinson's restatement of Alperin's weight conjecture \cite[Theorem~3.8]{kno89}, which in our language is equivalent to \(l(G,B)\), the number of irreducible Brauer characters in \(B\),  being block chain local. They also show that \(k(G,B)-l(G,B)\) is, in our language, block chain local \cite[Corollary~4.3]{kno89} and thus \(k(G,B)\), the number of irreducible ordinary characters in \(B\), being block chain local is dependant on Alperin's weight conjecture holding. We extend some of the ideas in their work to show that the Aplerin--McKay can also be restated in a similar way.



In order to do this we first introduce chains, an involution on the set of chains and normalising triples. These are all defined by Isaacs and Navarro \cite{isa20} but we define them here so that this paper can be read independently of theirs. We can then go on to define what we mean by block chain local and introduce the Alperin--McKay function. We then prove that this function is block chain local if and only if the conjecture holds. This proof method has much the same structure as \cite[Section~4]{isa20} but note that most results will be different as we are considering blocks.

The final section of this paper is dedicated to proving several other block-by-block versions of results from Isaacs and Navarro, namely sufficient conditions for functions to be block chain local and an alternate proof that the Alperin--McKay conjecture holds for blocks of defect one.

Throughout this text \(G\) will be a finite group with order divisible by a prime \(p\) and \(k\) will be an algebraically closed field of characteristic \(p\) unless stated otherwise.

\section{Chains of subgroups}
\label{sec;bcl}

Here we introduce what we mean by chains, an involution on the set of chains and normalising triples. These are all defined by Isaacs and Navarro \cite{isa20} but some of the results here differ from theirs.


\subsection{Chain stabilisers}

\begin{defn}
    Given a group \(G\) we say a {\it chain} \(C\) in \(G\) is a totally ordered set of \(p\)-subgroups, \( C = \set{Q_i}{0 \leq i \leq n}\), with ordering \({Q_0 < Q_1 < \cdots < Q_n}\) and where \(Q_0\) is the trivial group.
\end{defn}

We define the length of \(C\), denoted \(|C|\), to be the number of non-trivial groups in \(C\), in other words for \(C\) as above \(|C|=n\). Our requirement that \(Q_0=\{1\}\) means all our chains will have non-negative length. Note that this is not universal and other authors may not require \(\{1\}\) to be in the chain. 

These chains are the same as those used by Isaacs and Navarro \cite[p.~2]{isa20} and Kn\"{o}rr and Robinson \cite[Definition~2.1]{kno89}. Our choice to use chains of arbitrary \(p\)-subgroups is not universal and in fact a result by Kn\"{o}rr and Robinson \cite[Proposition~3.3]{kno89} shows that for many purposes normal chains, chains of radical subgroups and chains of elementary abelian subgroups can be used. 

\begin{defn}
    The {\it chain stabiliser}, \(G_C\), of a chain \(C\) in \(G\) is the set of elements of \(G\) that stabilise all elements in \(C\) under the conjugation action. In other words \(G_C\) is the intersection of the normalisers of each \(Q_i\) in \(C\).
\end{defn}

The chain stabiliser subgroups are what Isaacs and Navarro use to define chain local and with that obtain many results. As we are working with blocks we also need to introduce a way to check when a block of \(kG_C\) inducts up to a block of \(kG\).

First we introduce what we mean by block induction in this context. Let \(B\) be  a block of \(kG\) and let \(Q\) be a subgroup of \(G\) with \(b\) a block of \(kQ\). Then we say that \(b\) corresponds to a block \(B\) of \(kG\) if \(\Br{Q}{e_B}e_b = e_b\). In other words the block idempotent of \(b\) is a summand of the restriction of the block idempotent of \(B\) to the centraliser of \(Q\) in \(G\). We denote this as \(b^G=B\) and call it {\it block induction}. Note that it is not guaranteed that \(b\) has a block of \(kG\) corresponding to it. When \(b\) does have a corresponding block we say \(b^G\) is defined. Further details of ths definition of block induction can be found in \cite[Section 5.3]{nag89}. We can now recall some basic results relating to block induction.

\begin{lem}[see {\cite[Lemma 5.3.3]{nag89}}]
    \label{lem;blj}
    Let \(Q\) be a subgroup of \(G\) and \(b\) a block of \(kQ\) with defect group \(D\). If \(b^G\) is defined then \(D\) is contained in a defect group of \(b^G\).
\end{lem}

\begin{lem}[see {\cite[Lemma 5.3.4]{nag89}}]
    \label{lem;bli}
    Let \(Q\) and \(K\) be subgroups of \(G\) with \(b\) a block of \(Q\). If  \(Q\) is also a subgroup of \(K\) with the block \(b^K\) defined and either  \(\left(b^K\right)^G\) or \(b^G\) are defined then the third is also defined and \({b^G = \left(b^K\right)^G}\).
\end{lem}

We define \(B_C\) to be \(\Br{Q_n}{e_B} k G_C\), where \(e_B\) is the central idempotent corresponding to \(B\) and \(\text{Br}_{Q_n}\) is the Brauer morphism at \(Q_n\). Note that \(\Br{Q_n}{e_B}\) is either zero or is a central idempotent of \(kG_C\), as \({\ce{G}{Q_n} \leq G_C \leq \no{G}{Q_n}}\), and thus \(B_C\) is either zero or a sum of blocks of \(G_C\). This sum of blocks was defined by Kn\"{o}rr and Robinson \cite{kno89}, who proved the following result. 

\begin{lem}[{\cite[Lemma 3.2]{kno89}}]
    \label{lem;bcb}
    Let \(C\) be a chain in \(G\). If \(b\) is a block of \(kG_C\) we have that \(b^G\) is defined and further \(b^G = B\) if and only if \(b\) is a summand of \(B_C\).
\end{lem}

We write \(b \mid B_C\) if \(b\) is a summand of \(B_C\). Thus we can now link blocks of \(kG_C\) to blocks of \(G\).

\subsection{An involution on the set of chains}

Let \(P \leq G\) be a non-trivial \(p\)-subgroup. Consider the set of all chains of \(p\)-subgroups of \(G\)  where either \(P\) normalises every group in \(C\) or every group in \(C\) normalises \(P\), for each \(C\) in this set. We can define a permutation on the set of chains as follows. 

Let \( C = \set{Q_i}{0 \leq i \leq n}\) be in this set and note that as \(P \nsubseteq Q_0\) there is a maximal \(m\) such that \(0 \leq m \leq |C|\) and \(P \nsubseteq Q_m\). Thus \(Q_m < Q_m P\) and \(Q_m P \subseteq Q_i\) for all \(m+1 \leq i \leq n\). Now we consider two cases: if \(Q_m P\) is in \(C\), we set \(C^* = C \setminus \left\{ Q_m P\right\}\) and otherwise, if \(Q_m P\) is not in \(C\) we set \(C^* = C \cup \left\{ Q_m P\right\}\). 

The chain \(C^*\) depends on the group \(P\) as well as on \(C\), however the \(P\) in question will generally be clear from context so we do not include it in the notation for \(C^*\). We can now state a few basic results about this permutation.

\begin{lem}[{\cite[Lemma~4.1]{isa20}}]  \label{lem;sta}
    If \(C\), \(P\) and \(G\) are as above, then the following properties hold:
    \begin{enumerate}
        \item   \label{lem;sta1} Either \(C \subsetneq C^*\) and \(|C^*| = |C| + 1\) or \(C^* \subsetneq C\) and \(|C^*| = |C| - 1\).
        \item   \label{lem;sta2} \(\left( C^* \right)^* = C\).
        \item   \label{lem;sta3} Every subgroup of \(G\) contained in \(G_C\) and normalising \(P\) is also contained in \(G_C\).
    \end{enumerate}
\end{lem}

All parts of this lemma should be easy to see from our construction. 

\subsection{Normalising triples}

\begin{defn}
    Let \(C\) be a chain of \(p\)-subgroups of \(G\). Let \(P\) be a non-trivial \(p\)-subgroup of \(G_C\) and let \(X\) be a non-empty subset of \(G_C \cap \no{G}{P}\). We call \(\left( C, P, X \right)\) a {\it normalising triple}.
\end{defn}

Again these are first defined by Isaacs and Navarro \cite[p.~16]{isa20}. We have a well-defined action of \(G\) on the set of normalising triples given by conjugation of \(C\), \(P\) and \(X\) by some \(g\) in \(G\). Let \(\OO\) be an orbit of normalising triples under this action. All chains that form the left-hand component of triples in \(\OO\) have the same length so we can set \(s(\OO) = (-1)^{|C|}\), where \(C\) is the first component of a triple in \(\OO\). If \(\RR\) is a set of representatives of the orbits of chain of \(p\)-subgroups of \(G\) under conjugation then one can see that for each orbit \(\OO\) exactly one member of \(\RR\)  appears as the first component in a member of \(\OO\). 

Notice that as \(P\) is contained in \(G_C\) it can be used to define \(C^*\). As \(X\) is contained in \(\no{G}{P}\) and \(G_C\) then it is also in \(G_{C^*}\) and thus \(\left( C^*, P, X \right)\) is also a normalising triple uniquely determined by \(\left( C, P, X \right)\). This gives us a permutation on the set of normalising triples where \(\left( C, P, X \right)^* = \left( C^*, P, X \right)\).

We can extend the idea of this permutation as follows. If \(\OO\) is a \(G\)-orbit of normalising triples, then set
    \[
        \OO^* = \set{\left( C^*, P, X \right)}{\left( C, P, X \right) \in \OO}.
    \]
The set \(\OO^*\) is also an orbit of normalising triples as \(\left(\left( C, P, X \right)^g\right)^* = \left(\left( C, P, X \right)^*\right)^g\) for all \(g\) in \(G\). Additionally \(s(\OO) = - s(\OO^*)\) so \(\OO \neq \OO^*\) but \(|\OO| =|\OO^*|\). Finally note that \((\OO^*)^* = \OO\) and we say that \(\OO\) and \(\OO^*\) are paired. 

We now quote a further result by Isaacs and Navarro \cite[Lemma~4.4]{isa20} which describes sets of normalising triples.

\begin{lem}[{\cite[Lemma~4.4]{isa20}}]
    \label{lem;orb}
    For \(G\) a finite group, let \(\TT\) be a set of normalising triples in \(G\), invariant under the \(G\)-action. Let \(C\) be a chain in \(G\) and \(\QQ\) be the set of pairs \((Q,Y)\) such that \(\left( C,Q, Y \right)\) is in \(\TT\). Let \(\UU\) be a \(G_C\)-orbit on \(\QQ\) and let 
    \[
        \SS = \set{\left( C,Q, Y \right)}{(Q,Y) \in \UU}.
    \]
    Then there exists a unique \(G\)-orbit \(\OO\) on \(\TT\) such that \(\SS\) is contained in \(\OO\). Furthermore the map \(\tau\), where we set \(\tau(\UU) = \OO\), is a bijection from the set of \(G_C\)-orbits on \(\QQ\) onto the set of \(G\)-orbits \(\OO\) on \(\TT\) such that \(C\) is the first component of some triple in \(\OO\).
\end{lem}

Note that this makes sense because \(\QQ\) is \(G_C\)-invariant as well as \(G\)-invariant. The proof of this result is again due to Isaacs and Navarro \cite[Lemma~4.4]{isa20}. The following result gives us some properties of a normalising triple provided \({X=\no{G_C}{P}}\).

\begin{lem}
    \label{lem;bGB}
    Let \(G\) be a group and \(B\) a block of \(kG\). If \(\left( C,P, X \right)\) is a normalising triple where \(X= \no{G_C}{P}\), then the following hold:
        \begin{enumerate}
            \item   \(X=\no{G_{C^*}}{P}\).
            \item   If \(P\) is a defect group of the block \(b\) of \(kG_C\) with positive defect and \(b^G=B\) then there is a unique block \(b^*\) of \(G_{C^*}\) associated to \(b\) with defect group \(P\) and \((b^*)^G=B\).
        \end{enumerate}
\end{lem}

\begin{proof}
    Part (i) is due to Isaacs and Navarro \cite[Lemma~4.3(a)]{isa20}. For part (ii) first we note that as \(\no{G_{C}}{P}=\no{G_{C^*}}{P}\) we can apply the Brauer correspondence twice to get from blocks of \(G_C\) with defect group \(P\) to blocks of \(G_{C^*}\) with defect group \(P\). Thus the block of \(kG_{C^*}\) corresponding to \(b\) is the block \(b^*\) where \(\Br{P}{e_{b^*}}=\Br{P}{e_b}\) and we know this is unique by the Brauer correspondence. Also as either \(C^* \subsetneq C\) or \(C \subsetneq C^*\) we have that either \((b^*)^{G_C}=b\) or \(b^{G_{C^*}}=b^*\) again by the Brauer correspondence. Thus by transitivity of block induction, as stated in Lemma~\ref{lem;bli}, we have \((b^*)^G=B\).
\end{proof}

Notice that the reason we need to exclude blocks of defect zero in (ii) is because they have trivial defect group and normalising triples are not allowed to have trivial second component. 

\section{Block chain local functions}

In order to prove an equivalence for the Alperin--McKay conjecture we first needed to define a block-by-block version of chain local, which we have called block chain local. This section follows the general structure of \cite[Section~4]{isa20}, however the results themselves and their proofs are original, unless stated otherwise.

\subsection{Functions on blocks}

We will consider a collection of pairs \((G,B)\), where \(G\) is a finite group of order divisible by \(p\) and \(B\) is a block of \(kG\), which we will call a family, denoted \(\FF\). We require that if \((G,B)\) is in \(\FF\) then so are all pairs \((H,b)\) where \(H \leq G\) and \(b^G=B\). Let \(U\) be a free abelian group, which will be the integers unless otherwise stated. We can define maps from \(\FF\) to \(U\) and we require that if there exist pairs \((G,B)\) and \((H,B')\) in \(\FF\) such that there exists an isomorphism \(\phi: G \rightarrow H\) with \(\phi(B)=B'\) then each of these maps \(f\) must have the property \(f(G,B)=f(H,B')\), in other words they are isomorphism constant. Note that these pairs in general are not Brauer pairs.

Isaacs and Navarro \cite{isa20} considered a family consisting of just groups and functions on that family. From there they defined chain local and used it to show certain functions were locally determined. We wish to do something similar, however with functions on our family \(\FF\). In order to do this we define block chain local, a block-by-block version of chain local. Lemma~\ref{lem;bcb} allows us to determine which blocks of \(G_C\) induce up to \(G\), which will be very useful.

\begin{defn}
    An isomorphism-constant function \(f\) on \(\FF\) is {\it block chain local} if for all pairs \((G,B)\) in \(\FF\) we have 
    \[
        \sum_{C \in \RR} (-1)^{|C|} \sum_{b \mid B_C}f(G_C,b)=0,
    \]
    where \(\RR\) is a set of representatives of \(G\)-orbits of chains of \(p\)-subgroups of \(G\) and \(B_C\) is as defined above.
\end{defn}

We will also sometimes say a function is block chain local on just a pair \((G,B)\) in which case we are referring to the smallest family \(\FF\) containing \((G,B)\), in other words \((G,B)\) and all pairs \((H,b)\) where \(H \leq G\) and \(b^G=B\). We will call the double sum above the alternating chain sum. It is easy to see that the sum and integer multiples of block chain local functions with values in \(\Z\) are also block chain local. We now prove a further property of block chain local functions.

\begin{lem}
    \label{lem;f=g}
    Let \(f\) and \(g\) be block chain local functions on some family \(\FF\). If we have that \({f(H,b) = g(H,b)}\), for all subgroups \(H\) of \(G\) with \(O_p(H) > 1\) and blocks \(b\) of \(kH\) with \(b^G=B\), then \(f(G,B) = g(G,B)\).
\end{lem}

\begin{proof}
    We can assume that \(O_p(G)=1\) as otherwise the result holds trivially. As \(f\) and \(g\) are block chain local we have 
    \[
        \sum_{C \in \RR} (-1)^{|C|} \sum_{b \mid B_C}f(G_C,b) = 0 = \sum_{C \in \RR} (-1)^{|C|} \sum_{b \mid B_C}g(G_C,b),
    \]
    where \(\RR\) is a set of representatives of \(G\)-orbits of chains of \(p\)-subgroups. Notice that the trivial chain \(C_0\) is always in \(\RR\) and that for the trivial chain we have \(G_{C_0}=G\). When \(G_C=G\) we have \(B_C=B\) and thus \(B\) is trivially the only summand of \(B_{C_0}\). Also note that, as \(O_p(G)=1\), the trivial chain is the only such chain where \(G_C=G\). We can therefore rearrange the left-hand side of the above to be 
    \[
        f(G,B) = - \sum_{C \in \RR \setminus \{C_0\}} (-1)^{|C|} \sum_{b \mid B_C}f(G_C,b),
    \]
    as \(|C_0|=0\). We can also do the same with the right-hand side. Note that for \(C\) in \(\RR \setminus \{C_0\}\) we have \(O_p(G_C)>1\) so 
    \[
        \sum_{b \mid B_C}f(G_C,b) =  \sum_{b \mid B_C}g(G_C,b).
    \]
    Thus summing over all of \(\RR \setminus \{C_0\}\) we get 
    \[
        f(G,B) = - \sum_{C \in \RR \setminus \{C_0\}} (-1)^{|C|} \sum_{b \mid B_C}f(G_C,b) = - \sum_{C \in \RR \setminus \{C_0\}} (-1)^{|C|} \sum_{b \mid B_C} g(G_C,b) = g(G,B).
    \]
\end{proof}

We can also, using our involution on chains, prove a strong result about all conjugacy-constant functions on pairs.

\begin{prop}
    \label{pro;opg}
    If \(G\) is a finite group such that \(O_p(G) >1\) and \(B\) is a block of \(kG\), then every isomorphism-constant function on \((G,B)\) and pairs \((H,b)\), where \(H\) is a subgroup of \(G\) and \(b^G=B\), is block chain local.
\end{prop}

\begin{proof}
    If \(Q= O_p(G) >1\), then for any chain \(C\) in \(G\) we have that \(Q\) is normalised by every group of \(C\) and thus \(C^*\) is defined with respect to \(Q\). We therefore get a bijection from the set of chains of odd-length to the set of chains of even-length given by sending \(C\) to \(C^*\). Each of these sets is invariant under \(G\)-conjugation. Also note that we have \(\left(C^g\right)^* = \left(C^*\right)^g\) for all \(g\) in  \(G\) and thus \(G_{C^*}=G_{C}\). Further, we must also have \(B_{C^*}=B_{C}\).
    
    Let \(\RR\) be a set of representatives of \(G\)-orbits of even-length chains in \(G\) and notice that we have \({\RR^* = \set{C^*}{C \in \RR}}\). Thus for \(f\) an isomorphism-constant function on \((G,B)\) and pairs \((H,b)\), where \(H\) is a subgroup of \(G\) and \(b^G=B\), we can write the alternating chain sum as 
    \[
        \sum_{C \in \RR} \left((-1)^{|C|} \sum_{b \mid B_C}f(G_C,b) + (-1)^{|C^*|} \sum_{b \mid B_{C^*}}f(G_{C^*},b) \right).
    \]
    Thus, because \(|C^*|=|C|\pm 1\), the pairs of terms cancel and the sum above goes to zero. We therefore have that \(f\) is block chain local on \((G,B)\) and pairs \((H,b)\), where \(H\) is a subgroup of \(G\) and \(b^G=B\).
\end{proof}

This means from now on when proving an isomorphism-constant function is chain local we need only consider groups \(G\) in pairs in \(\FF\) where \(O_p(G)=1\).

\subsection{The Brauer correspondence and block chain local functions}

In this section we will explore a specific condition for when a function is block chain local, arising from the Brauer correspondence. To do this we define several functions on pairs and show that they are block chain local. The proof of this result will follow roughly the same structure as the proof of Theorem~D by Isaacs and Navarro \cite[Section~4]{isa20}.

Fix a group \(G\) and a block \(B\) of \(kG\) with positive defect. Let \(P\) be a non-trivial \(p\)-subgroup of \(G\) and \(X\) a subset of \(G\). For all subgroups \(H\) of \(G\) and blocks \(b\) of \(kH\) with \(b^G=B\), define \(\omega_{(P,X)} (H,b)\) to be the number of \(H\)-orbits of pairs \((Q,Y)\) that are \(G\)-conjugate to \((P,X)\), where \(Q\) is a subgroup of \(H\), \(Y=\no{H}{Q}\) and \(Q\) is a defect group of \(b\).

\begin{prop}
    \label{pro;omp}
    Let \(G\) be a group and \(B\) a block of \(kG\) with positive defect. The function \(\omega_{(P,X)}\) is block chain local on \((G,B)\).
\end{prop}

\begin{proof}
    Consider \(H=G_C\), the stabiliser of a chain \(C\) of \(p\)-subgroups in \(G\). By definition \(Y\) normalises \(Q\) and both \(Y\) and \(Q\) are subsets of \(G_C\). Thus \((C,Q,Y)\) is a normalising triple. Let \(\TT\) be the set of normalising triples \((C,Q,Y)\) such that \((Q,Y)\) is \(G\) conjugate to \((P,X)\), \(Y = \no{G_C}{Q}\) and \(Q\) is a defect group of \(b\), where \(b\) is a block of \(kG_C\) and \(b^G=B\). For a triple \((C,Q,Y)\) in \(\TT\) note that \((C,Q,Y)^*\) is also in \(\TT\) by Lemma~\ref{lem;bGB}(ii). By their definitions one can see that \(\omega_{(P,X)} (G_C,b)\) is the number of \(G_C\)-orbits of pairs \((Q,Y)\) such that \((C,Q,Y)\) lies in \(\TT\) and \(Q\) is a defect group of \(b\). Thus we see that 
    \[
        \sum_{b \mid B_C} \omega_{(P,X)} (G_C,b),
    \]
    is the number of \(G_C\)-orbits of pairs \((Q,Y)\) such that \((C,Q,Y)\) lies in \(\TT\). We can therefore apply Lemma~\ref{lem;orb} to show that this is also the number of \(G\)-orbits \(\OO\) of elements of \(\TT\) such  that \(C\) appears as the first component of a triple in \(\OO\). Recall that, as \(C\) is the first component of a triple in the \(G\)-orbit \(\OO\), by definition \((-1)^{|C|} = s(\OO)\) and we have 
    \[
        (-1)^{|C|} \sum_{b \mid B_C} \omega_{(P,X)} (G_C,b) = \sum_\OO s(\OO),
    \]
    where the sum on the right is over all \(G\)-orbits \(\OO\) in \(\TT\) that contain a triple with first component \(C\). If we now consider the set \(\RR\) of representatives of \(G\)-orbits of chains in \(G\) then for each \(\OO\) in \(\TT\) there is exactly one \(C\) in \(\RR\) and \(b\) in \(B_C\) such that \(C\) is the first component of a triple in \(\OO\). Thus if we sum over \(\RR\) we have 
    \[
        \sum_{C \in \RR}(-1)^{|C|} \sum_{b \mid B_C}\omega_{(P,X)} (G_C,b) = \sum_\OO s(\OO),
    \]
    where now the sum on the right is over all \(G\)-orbits \(\OO\) in \(\TT\). As we have established, every \(G\)-orbit \(\OO\) in \(\TT\) is paired with a unique \(G\)-orbit \(\OO^*\) which is also in \(\TT\). Thus, as \(s(\OO)=-s(\OO^*)\) we see that the above sum is zero and \(\omega_{(P,X)}\) is block chain local on the arbitrary pair \((G,B)\), where \(B\) has positive defect.
\end{proof}

We can now define the next function. Let \(N\) and \(G\) be finite groups and \(B\) a block of \(kG\) with nontrivial defect group \(P\). Define \(\Omega_N (G,B) = 1\) if \(\no{G}{P} \cong N\) and \({\Omega_N (G,B) = 0}\) otherwise.

\begin{cor}
    \label{cor;omn}
    The function \(\Omega_N\) is block chain local.
\end{cor}

\begin{proof}
    First note that we can assume \(O_p(N)\) is non-trivial without loss of generality as, when \(O_p(N)\) is trivial \(\Omega_N\) is identically zero and hence trivially block chain local. Let \((G,B)\) be a pair in \(\FF\) where \(B\) has nontrivial defect group \(P\). Let \(\SS\) be a set of pairs of subgroups \((P,X)\) of \(G\) such that \(P \nleq X\) and \(X \cong N\). Let \(\omega\) denote the restriction of \(\Omega_N\) to pairs \((H,b)\), where \(H\) is a subgroup of \(G\) and \(b^G=B\). We proceed by showing all the restrictions are block chain local for arbitrary \((G,B)\) in \(\FF\) which in turn will show that \(\Omega_N\) is block chain local. 
    
    Let \(H\) be a subgroup of \(G\) with order divisible by \(p\) and \(b\) be a block of \(kH\) such that \(b^G=B\). By definition we have \(\omega (H,b)=1\) if \(\no{H}{Q} \cong N\) where \(Q\) is a defect group of \(b\) and \(\omega (H,b)=0\) otherwise. Since all defect groups of \(b\) are \(H\)-conjugate, \(\omega (H,b)\) is the number of \(H\)-orbits on the set \(\PP\) of pairs \((Q,Y)\) such that \(Y=\no{H}{Q}\), \(Q\) is a defect group of \(b\) and \(Y \cong N\). Further, every member of each \(H\)-orbit on \(\PP\) is \(G\)-conjugate to some unique \((P,X)\) in \(\SS\). The total number of \(H\)-orbits on \(\PP\) whose members are \(G\)-conjugate to some given member \((P,X)\) of \(\SS\) is \(\omega_{(P,X)} (H,b)\), where \(\omega_{(P,X)}\) is as in Proposition~\ref{pro;omp}. Thus 
    \[
        \omega (H,b) = \sum_{(P,X) \in \SS} \omega_{(P,X)} (H,b),
    \]
    and as \(\omega_{(P,X)}\) is block chain local \(\omega\) is as well. Thus as \(\omega\) is block chain local and, as \((G,B)\) in \(\FF\) was arbitrary, we have that \(\Omega_N\) is block chain local for any \(N\) a finite group.
\end{proof}

\begin{thm}
    \label{thm;fgn}
    Let \(f\) be an isomorphism-constant function on a family \(\FF\) with values in a free abelian group. If, for all \((G,B)\) in \(\FF\), we have \(f(G,B)=f(N,b)\), where \(N = \no{G}{P}\), \(P\) is a defect group of \(B\) and \(b\) is the Brauer correspondent of \(B\), then \(f\) is block chain local for all pairs \((G,B)\) where \(B\) has positive defect.
\end{thm}

\begin{proof}
    We make use of the function \(\Omega_M\), which we proved is block chain local for any finite group \(M\) in Corollary~\ref{cor;omn}. Let \((G,B)\) be a pair in \(\FF\) where \(B\) has positive defect. Let \(N\) be the normaliser of a defect group of \(B\), so \((N,B')\) is in \(\FF\), where \(B'\) is the Brauer correspondent of \(B\), and \(f(N,B')\) is defined. We have that \(f(G,B)=f(N,B')\). However note that we can also write
    \[
        f(N,B') = \sum_{[M]} \Omega_M(G,B) \sum_{b'} f(M,b'),
    \]
    where the left sum runs over all isomorphism classes \([M]\) of groups in pairs in \(\FF\) and the right sum over all blocks \(b'\) of \(M\) such that \(b'^G=B\). This is because \(\Omega_M(G,B)=0\) unless \(M \cong N\) in which case \(\Omega_M(G,B)=1\) and the only block of \(kM\) where \(b'^G=B\) is \(B'\).
    
    If we now consider a chain \(C\) in \(G\) with stabiliser \(G_C\) and \(b\) a block of \(kG_C\) we have 
    \[
        f(G_C,b) = \sum_{[M]} \Omega_M(G_C,b) \sum_{b'} f(M,b'),
    \]
    where the left sum is as above and the right sum is over all \(b'\) of \(M\) where \(b'^{G_C}=b\). Now, for \(\RR\) a set of representative of \(G\)-orbits of chains, we have 
    \begin{align*}
        \sum_{C \in \RR} (-1)^{|C|} \sum_{b \mid B_C} f(G_C,b)   & = \sum_{C \in \RR} (-1)^{|C|} \sum_{b \mid B_C}\sum_{[M]} \Omega_M(G_C,b) \sum_{b'} f(M,b') \\
                                                                & = \sum_{[M]} \left(\sum_{C \in \RR} (-1)^{|C|} \sum_{b \mid B_C} \Omega_M(G_C,b) \right)\sum_{b'} f(M,b') \\
                                                                & = 0,
    \end{align*}
    as \(\Omega_M\) is block chain local for all \(M\) a finite group. Thus as \(G\) and \(B\) were arbitrary we have that \(f\) is block chain local on all pairs \((G,B)\) where \(B\) has positive defect.
\end{proof}

\subsection{The Alperin--McKay function}

Recall the Alperin--McKay conjecture states that for a group \(G\) the number of irreducible height zero characters in a block \(B\) of \(kG\) with defect group \(D\) is the same as the number of irreducible height zero characters of the block \(B'\) of \(\no{G}{D}\) with defect group \(D\), that is Brauer correspondent to \(B\). We can rephrase this in terms of a function on a family of pairs \(\FF\). Let \(\text{am}\) be the Alperin--McKay function given as 
    \[
        \am{G}{B} = \text{ \# of irreducible height zero characters of \(B\)},
    \]
for \((G,B)\) a pair in \(\FF\). This function is obviously isomorphism-constant and maps to the integers. The Alperin--McKay conjecture is the statement \(\am{G}{B} = \am{N}{B'}\), where \(N = \no{G}{D}\) and \(D\) is the defect group of \(B\) and \(B'\) is its Brauer correspondent, for all \((G,B)\) in the family \(\FF\) and \(\FF\) being arbitrary. We now have everything necessary to state our motivating result.

\begin{thm}
    \label{thm;amf}
    The Alperin--McKay conjecture holds if and only if the Alperin--McKay function is block chain local on all pairs \((G,B)\) such that \(B\) has positive defect.
\end{thm}

Note we do not need to consider blocks of defect zero as the conjecture is vacuous for blocks with normal defect group, in particular defect zero.

\begin{proof}
    The forward direction is a clear result of Theorem~\ref{thm;fgn}: recall that the Alperin--McKay conjecture can be stated as \(\am{G}{B}=\am{N}{B'}\) for all \(G\) a \(p\)-group and all \(B\) a block of \(kG\) where \(N = \no{G}{P}\), \(P\) is a defect group of \(B\) and \(B'\) is the Brauer correspondent of \(B\). 
    
    Now let us consider the other direction, in other words suppose that the Alperin--McKay function is block chain local and show this implies the Alperin--McKay conjecture holds. We do this by contradiction. Let \((G,B)\) be a minimal counterexample to the conjecture, so \(\am{G}{B} \neq \am{N}{B'}\). This implies that \(N \neq G\) and thus \(|G|\) is divisible by \(p\). By a result of Kessar and Linckelmann \cite[Proposition~5]{kes19} we have that if \((G,B)\) is a minimal counterexample to the Alperin--McKay conjecture then \(O_p(G)=1\), so we can assume this as well.
    
    Let \(\text{am}_0\) be the function defined by setting \(\text{am}_0 (X,b) = \am{M}{b'}\), where \(M\) is the normaliser of the defect group of \(b\) and \(b'\) is the Brauer correspondent of \(b\). This is clearly isomorphism-constant. By definition \(\text{am}_0 (X,b) = \am{M}{b'} = \text{am}_0 (M,b')\), for all \((X,b)\) in \(\FF\), so by Theorem~\ref{thm;fgn} \(\text{am}_0\) is chain local.
    
    Note for \(X \subseteq G\) with \(O_p(X) >1\) this implies \(X\) is a proper subgroup of \(G\) as \(O_p(G)=1\). As \(G\) is minimal we have \(\am{X}{b} = \am{M}{b'} = \text{am}_0 (X,b)\), where \(b\) is  a block of \(kX\) such that \(b^G=B\) and \(M\) and \(b'\) are as before. Thus, by Lemma~\ref{lem;f=g}, we have that \(\am{G}{B} = \text{am}_0 (G,B) = \am{N}{B'}\), where \(N\) is the normaliser of a defect group of \(B\) and \(B'\) is a block of \(kN\) such that \(B'^G=B\). This is a contradiction and the result is proved.
\end{proof}

The proof in itself relies on the insightful result of Kessar and Linckelmann \cite[Proposition~5]{kes19} as without this information about a minimal counterexample our method would not work.

\subsection{Sufficient conditions for block chain local}

Here we collect several more results that are also adapted from those given by Isaacs and Navarro. We start with a direct corollary of Theorem~\ref{thm;fgn}.

\begin{cor}
    \label{cor;con}
    A constant function \(f\) is block chain local on all pairs \((G,B)\) where \(B\) has positive defect.
\end{cor}

\begin{proof}
    As \(f\) is constant for all pairs \((G,B)\) in \(\FF\) we have \(f(G,B)=f(N,b)\) where \(b\) is the Brauer correspondent of \(B\) and \(N\) is the normaliser of a defect group of \(B\). Thus Theorem~\ref{thm;fgn} applies and the result follows.
\end{proof}

The next two results tell us about functions that are entirely described by the values they take on \(p\)-subgroups and their blocks. 

\begin{thm}
    \label{thm;ngq}
    Let \(f\) be a function from some family \(\FF\) to a free abelian group \(U\) such that 
    \[
    f(G,B) = \sum_Q \sum_b h(\no{G}{Q},b),
    \]
    where the first sum runs over a set of representatives of \(G\)-orbits of \(p\)-subgroups, the second sum runs over all blocks \(b\) of \(\no{G}{Q}\) such that \(b^G=B\) and the function \(h\) is an isomorphism-constant function on \(\FF\). Then \(f\) is block chain local.
\end{thm}

This theorem is the block chain local version of Theorem~F(b) in Isaacs and Navarro \cite{isa20}. Our proof follows the same structure as theirs and thus in order to prove it we need to introduce one of the functions that they use. 

Let \(N\) be a finite group. Then for any group \(G\) let \(g_N(G)\) be the number of conjugacy classes of nontrivial \(p\)-subgroups \(Q\) of \(G\) such that \(\no{G}{Q} \cong N\). Isaacs and Navarro prove that this function is chain local \cite[Corollary~4.6]{isa20}, in other words we have 
\[
    \sum_{C \in \PP} g_N (G_C) =0,
\]
where \(G\) is any group with order divisible by \(p\) and \(\PP\) is a set of representatives of \(G\)-orbits of chains in \(G\).

\begin{proof}
    Using the function \(g_N\) from above we see that for all pairs \((G,B)\) in \(\FF\) we can write \(f\) as
    \[
        f(G,B) = \sum_{[N]} g_N(G) \sum_{b'} h(N,b'),
    \]
    where the first sum runs over isomorphism classes of normalisers of nontrivial \(p\)-subgroups of \(G\) and the second sum runs over all blocks \(b\) of \(\no{G}{Q}\) such that \(b^G=B\). If we now consider the alternating chain sum of the arbitrary pair \((G,B)\) then we get 
    \begin{align*}
        \sum_{C \in \PP} \sum_{ b \mid B_C} f(G_C,b)   = & \sum_{C \in \PP} \sum_{ b \mid B_C} \sum_{[N]} g_N(G_C) \sum_{b'} h(N,b'), \\
                                                    = & \sum_{[N]} \sum_{C \in \PP} g_N(G_C) \sum_{ b \mid B_C} \sum_{b'} h(N,b'),
    \end{align*}
    where are sums are defined as before. As \(g_N\) is chain local \cite[Corollary~4.6(b)]{isa20} we have that \(\sum_{C \in \PP} g_N(G_C)=0\) for all groups \(N\). Thus for each \([N]\) in the first sum everything is zero and so 
    \[
        \sum_{C \in \PP} \sum_{ b \mid B_C} f(G_C,b) = 0,
    \]
    for all pairs \((G,B)\) in \(\FF\).
\end{proof}

This gives us a further sufficient condition to determine when a function is block chain local along with Theorem~\ref{thm;fgn}. There is a third such result we can adapt from Isaacs and Navarro \cite{isa20}, that is Theorem~F(c). First recall that a radical \(p\)-subgroup of a finite group \(G\) is a \(p\)-subgroup \(Q\) with the property \(Q = O_p(\no{G}{Q})\).

\begin{thm}
    Let \(f\) be a function from some family \(\FF\) to a free abelian group \(U\) such that 
    \[
    f(G,B) = \sum_Q \sum_b h(\no{G}{Q},b),
    \]
    where the first sum runs over a set of representatives of \(G\)-orbits of radical \(p\)-subgroups, the second sum runs over all blocks \(b\) of \(\no{G}{Q}\) such that \(b^G=B\) and the function \(h\) is an isomorphism-constant function on \(\FF\). Then \(f\) is block chain local.
\end{thm}

The proof of this result has exactly the same structure as that of Theorem~\ref{thm;ngq} however instead of using the function \(g_N\) it uses \(r_N\), also from Isaacs and Navarro \cite{isa20}. For \(N\) a finite group and \(G\) a group with order divisible by \(p\) we define \(r_N(G)\) to be the number of conjugacy classes of nontrivial radical \(p\)-subgroups \(Q\) of \(G\) such that \(\no{G}{Q} \cong N\). This function is also chain local \cite[Corollary~4.6(c)]{isa20} and thus the proof follows in the same way.

\subsection{Blocks with defect one}

Here we offer a block-by-block version of a result by Isaacs and Navarro  \cite[Theorem~E(d)]{isa20}. Let \(\FF\) be a family of pairs of the form \((G,B)\) and let \(k_1(G,B)\) be the number of irreducible ordinary characters in \(B\) with defect one.

\begin{thm}
    The function \(k_1\) is block chain local.
\end{thm}

\begin{proof}
    First note that for \(B\) a block of defect zero it cannot contain irreducible ordinary characters of defect one by definition, so \(k_1(G,B)=0\). By Lemma~\ref{lem;blj} all blocks \(b\) of subgroups of \(G\), where \(b^G=B\), have defect less than or equal to \(B\). Thus for all pairs \((G_C,b)\) in the alternating chain sum for \((G,B)\) we have that \(b\) has defect zero so \(k_1(G_C,b)=0\). The alternating chain sum is therefore zero and \(k_1\) is block chain local on pairs where the block has defect zero. Additionally Brauer showed that if a block contains an irreducible ordinary character of defect one then all irreducible ordinary characters have defect at most one \cite[Theorem~3]{bra41}. Thus for all blocks \(B\) with defect greater than one we get \(k_1(G,B)=0\) and, by Corollary~\ref{cor;con}, \(k_1\) is block chain local in this case. 
    
    We now only need to consider blocks \(B\) of \(kG\) with defect one as these are the only blocks where irreducible ordinary characters of defect one lie. Also only irreducible ordinary characters of defect one lie in these blocks. Let \(D\) be the defect group of \(B\). A result by Dade \cite[Theorem~1(i)]{dad66} gives the number of irreducible ordinary characters of \(B\) in terms of \(\no{G}{D}\) and the block \(b\) of \(\no{G}{D}\) with defect group \(D\). Crucially if we apply this to the block \(b\) we see that the number of irreducible ordinary characters of it are in terms of \(\no{G}{D}\) and the block \(b\). Thus we have \(k_1(G,B) = k_1(\no{G}{D},b)\), and noticing that \(B\) and \(b\) are Brauer correspondents we can apply Theorem~\ref{thm;fgn}. Therefore \(k_1\) is block chain local on pairs \((G,B)\) where \(B\) has defect one and therefore is for all pairs.
\end{proof}

\section{Acknowledgements}

Firstly I would like to thank my supervisor, Dr David Craven, for introducing me to this problem and for helping me throughout my work. I would also like to thank the School of Mathematics at the University of Birmingham and EPSRC for funding this research.



\bibliographystyle{maa}
\bibliography{main}{}

\begin{thebibliography}{1}
\expandafter\ifx\csname urlstyle\endcsname\relax
 \providecommand{\doi}[1]{doi:\discretionary{}{}{}#1}\else
 \providecommand{\doi}{doi:\discretionary{}{}{}\begingroup
  \urlstyle{rm}\Url}\fi

\bibitem{bra41}
Brauer, R. (1941).
\newblock Investigations on group characters.
\newblock \emph{Ann. of Math. (2)}, 42(4): 936--958.
\newblock \doi{https://doi.org/10.2307/1968775}.

\bibitem{dad66}
Dade, E.~C. (1966).
\newblock Blocks with cyclic defect groups.
\newblock \emph{Ann. of Math. (2)}, 84(1): 20--48.
\newblock \doi{https://doi.org/10.2307/1970529}.

\bibitem{isa20}
Isaacs, I.~M., Navarro, G. (2020).
\newblock Local functions on finite groups.
\newblock \emph{Represent. Theory}, 24: 1--37.
\newblock \doi{10.1090/ert/535}.

\bibitem{kes19}
Kessar, R., Linckelmann, M. (2019).
\newblock Dade’s ordinary conjecture implies the {A}lperin--{M}c{K}ay
  conjecture.
\newblock \emph{Arch. Math. (Basel)}, 112(1).
\newblock \doi{https://doi.org/10.1007/s00013-018-1230-9}.

\bibitem{kno89}
Knörr, R., Robinson, G. (1989).
\newblock Some remarks on a conjecture of {A}lperin.
\newblock \emph{J. Lond. Math. Soc. (2)}, 39(1): 48--60.
\newblock \doi{https://doi.org/10.1112/jlms/s2-39.1.48}.

\bibitem{mal22}
Malle, G., Navarro, G., Schaeffer~Fry, A.~A., Tiep, P.~H. (2022).
\newblock Brauer's height zero conjecture.
\newblock \doi{10.48550/ARXIV.2209.04736}.

\bibitem{nag89}
Nagao, H., Tsushima, Y. (1989).
\newblock \emph{Representations of Finite Groups}.
\newblock San Diego, CA: Academic Press.

\end{thebibliography}

\end{document}